\documentclass[11pt]{amsart}
\usepackage[T1]{fontenc}
\usepackage{lmodern}
\usepackage{amsmath,amssymb,amsfonts,amsthm,mathtools}
\usepackage{enumitem}
\usepackage[margin=1in]{geometry}
\usepackage{microtype}
\usepackage{xcolor}
\usepackage[colorlinks=true,citecolor=blue,linkcolor=blue,urlcolor=blue]{hyperref}
\newcommand{\F}{\mathbb F}
\newcommand{\Fq}{\F_q}
\newcommand{\Fqx}{\F_q^\times}
\newcommand{\Fqp}{\F_q^+}
\newcommand{\Fqm}{\F_q^-}
\newcommand{\sgnpat}{\boldsymbol\sigma}

\theoremstyle{plain}
\newtheorem{theorem}{Theorem}[section]
\newtheorem{lemma}[theorem]{Lemma}
\newtheorem{proposition}[theorem]{Proposition}
\newtheorem{corollary}[theorem]{Corollary}

\theoremstyle{remark}
\newtheorem{remark}[theorem]{Remark}

\title{Positivity preservers over finite fields II}

\author[D.~Guillot]{Dominique Guillot}
\address[D.~Guillot]{University of Delaware, Newark, DE, USA and Universit\'e Laval, Qu\'ebec, QC, Canada}
\email{\tt dguillot@udel.edu}

\author[H.~Gupta]{Himanshu Gupta}
\address[H.~Gupta]{University of Regina, Regina, SK, Canada}
\email{\tt himanshu.gupta@uregina.ca}

\author[P.K.~Vishwakarma]{Prateek Kumar Vishwakarma}
\address[P.K.~Vishwakarma]{Universit\'e Laval, Qu\'ebec, QC, Canada}
\email{\tt prateek-kumar.vishwakarma.1@ulaval.ca,~prateekv@alum.iisc.ac.in}

\author[C.H.~Yip]{Chi Hoi Yip}
\address[C.H.~Yip]{Hong Kong University of Science and Technology, Hong Kong}
\email{machyip@ust.hk}


\begin{document}
\emergencystretch=2em

\begin{abstract}
We say that a matrix over a finite field $\Fq$ is positive definite if it is symmetric and each of its leading principal minors is a nonzero square in $\Fq$. In previous work of the authors [\textit{J.~Algebra}, 2025], the entrywise positivity preservers on $M_n(\Fq)$ were classified in every case except when $n=2$, $q\equiv1\pmod4$, and $q$ is not a square. We settle this remaining case, thereby completing the classification of entrywise positivity preservers over every finite field and in every dimension $n\ge2$. Our proof is based on a novel idempotent reduction that not only resolves the remaining case but also yields a self-contained proof of the complete classification, while avoiding several technical results used in the earlier arguments.

As a further application of the same reduction, we classify the entrywise preservers of strongly nonsingular matrices, i.e., matrices whose leading principal minors are all nonzero. We also prove a more general theorem in odd characteristic: for every prescribed sign pattern of nonzero leading principal minors of matrices of a fixed dimension $n\ge2$, the entrywise preservers are precisely the positive scalar multiples of field automorphisms. Thus, in odd characteristic, preserving any nonzero leading-principal-minor sign pattern surprisingly forces the preservation of every such sign pattern.
\end{abstract}

\maketitle

\section{Introduction}

Let $q$ be a prime power and let $\Fq$ denote the finite field with $q$ elements. Write
\[
  \Fqx=\Fq\setminus\{0\},\qquad
  \Fqp=\{x^2:x\in\Fqx\},\qquad
  \Fqm=\Fqx\setminus\Fqp.
\]
We call the elements of $\Fqp$ \emph{positive}. A matrix $A\in M_n(\Fq)$ is \emph{positive definite} if it is symmetric and all its leading principal minors belong to $\Fqp$; see \cite{AyyerPrasad,CooperHannaWhitlatch,GGVY} and the references therein. For a function $f:\Fq\to\Fq$ and a matrix $A=(a_{ij})$, write $f[A]=(f(a_{ij}))$. We say that $f$ is an \emph{entrywise positivity preserver} on $M_n(\Fq)$ if $f[A]$ is positive definite whenever $A$ is positive definite.

The classification of entrywise functions preserving positivity is a classical problem in matrix analysis, with roots going back to P\'olya and Szeg\H{o}; see the surveys \cite{BGKP1,BGKP2} and the monograph \cite{KhareBook}. The finite-field problem was introduced in \cite{GGVY} by the authors as an algebraic analogue of Schoenberg's theorem \cite{Schoenberg1942}. There, the classification of entrywise positivity preservers over finite fields was completed in every dimension $n\ge3$, where it was shown that the preservers are precisely the positive multiples of field automorphisms. The dimension-two problem was also settled when $q$ is even, when $q\equiv3\pmod4$, and when $q$ is an odd square. The only case left open was
\[
  n=2,\qquad q\equiv1\pmod4,\qquad q\text{ not a square}.
\]

We settle this remaining case and, at the same time, obtain a self-contained and substantially more uniform proof of the complete classification. The proofs in \cite{GGVY} use different methods in the various cases. In even characteristic, one first proves in dimension two that the preserver has the form $x\mapsto cx^m$ and then uses a separate counting argument in higher dimension. In odd characteristic, the dimension-two arguments are more delicate. When $q\equiv3\pmod4$, the fact that $-1\in \F_q^-$ is used together with a character-sum estimate based on Weil's bound \cite{Weil} to prove injectivity on $\Fqm$; Carlitz's theorem \cite{Carlitz} then identifies the normalized preserver; see \cite[Lemma~4.1 and Theorem~4.3]{GGVY}. This argument does not extend directly to $q\equiv1\pmod4$, where $-1\in \F_q^+$.

For $q\equiv1\pmod4$, the proof in \cite{GGVY} exploits the connection between positivity preservers and Paley graphs. Recall that the Paley graph $P(q)$ has vertex set $\Fq$, with $x$ and $y$ adjacent when $x-y\in\Fqp$. Let $\Gamma(q)$ denote the local Paley graph, that is, the subgraph of $P(q)$ induced by $\Fqp$. In \cite[Proposition~5.8]{GGVY}, injectivity of a normalized preserver on $\Fqp$, together with positivity preservation and the strong regularity of $P(q)$, shows that its restriction to $\Fqp$ is an automorphism of $\Gamma(q)$. The Muzychuk--Kov\'acs theorem \cite{MuzychukKovacs} and further character-sum estimates then determine the preserver on the whole field $\F_q$. Furthermore, when $q=r^2$ is an odd square, \cite[Section~6]{GGVY} obtains the required injectivity from cosets of the subfield $\F_r$ via structural results on maximum and maximal cliques in $P(r^2)$ \cite{Blokhuis,GoryainovEtAl}; see also \cite{YipPaleyII}. For Paley graphs of nonsquare order, however, such structural results are unavailable. In
fact, determining the asymptotic behavior of their clique number remains a well-known open problem \cite{CrootLev,YipClique}. This is the main obstacle to extending the argument of \cite{GGVY} from the square case to all $q\equiv 1 \pmod 4$. 

In this paper, we bypass this technical barrier by considering iterates of positivity preservers and reducing the problem to a much more rigid special case. The resulting proof is uniform in odd characteristic and uses only basic properties of finite fields. In particular, it avoids Carlitz's theorem, the Muzychuk--Kov\'acs theorem, structural results on cliques in Paley graphs, and Weil bounds for character sums.

We can now state our main results.

\begin{theorem}\label{thm:positive-main}
Let $q=p^k$ be a prime power, let $n\ge2$, and let $f:\Fq\to\Fq$.
Then $f$ preserves positive definiteness on $M_n(\Fq)$ precisely in
the following cases.
\begin{enumerate}[label=\textup{(\arabic*)}]
\item If $n=2$ and $q$ is even, then
\[
  f(x)=cx^m
\]
for some $c\in\Fqx$ and $1\le m\le q-1$ with $\gcd(m,q-1)=1$.
\item If $n\ge3$ and $q$ is even, then
\[
  f(x)=cx^{2^j}
\]
for some $c\in\Fqx$ and $0\le j\le k-1$.
\item If $n\ge2$ and $q$ is odd, then
\[
  f(x)=cx^{p^j}
\]
for some $c\in\Fqp$ and $0\le j\le k-1$.
\end{enumerate}
\end{theorem}

Thus, in odd characteristic there is no exceptional dimension-two
phenomenon: the preservers are the same positive multiples of field
automorphisms for every $n\ge2$. In even characteristic, dimension two
is genuinely different: every map $x\mapsto cx^m$ with $c\ne0$ and
$\gcd(m,q-1)=1$, i.e., every nonzero injective monomial, is allowed.

Our proof is based on a novel idempotent reduction. The key point is that the identity is the only idempotent preserver of the relevant classes of $2\times2$ matrices. Since $\Fq$ is finite, this implies that every
preserver is a permutation and that its inverse has the same preservation
property. The fact that the inverse is again a preserver allows us to determine when certain determinants
vanish and to derive the functional identities needed for the classification. Some ideas in our proof are inspired by related arguments in
\cite{GGVY,GGVYSign}.

The same idea also yields a classification of entrywise preservers of strong nonsingularity. We say that a 
matrix $A\in M_n(\Fq)$ is \emph{strongly nonsingular} if all of its leading
principal minors are nonzero. The analogous notion for real and complex matrices is well-studied and arises naturally in the study of Gaussian elimination and LU factorization. See, e.g., \cite[Theorem 3.2.1]{GolubVanLoan} and \cite[Chapter 4]{VandebrilVanBarelMastronardi} for more details. We say that $f:\Fq\to\Fq$ \emph{preserves strong nonsingularity} if $f[A]$ is strongly nonsingular whenever $A\in M_n(\F_q)$ is strongly nonsingular.

\begin{theorem}\label{thm:strong-main}
Let $q=p^k$ be a prime power, let $n\ge2$, and let $f:\Fq\to\Fq$.
The following are equivalent.
\begin{enumerate}[label=\textup{(\roman*)}]
\item $f$ preserves strong nonsingularity on $M_n(\Fq)$.
\item $f[A]$ is strongly nonsingular whenever $A\in M_n(\Fq)$ is symmetric
and strongly nonsingular.
\item One of the following holds:
\begin{enumerate}[label=\textup{(\arabic*)}]
\item if $n=2$, then
\[
  f(x)=cx^m
\]
for some $c\in\Fqx$ and $1\le m\le q-1$ with $\gcd(m,q-1)=1$.
\item if $n\ge3$, then
\[
  f(x)=cx^{p^j}
\]
for some $c\in\Fqx$ and $0\le j\le k-1$.
\end{enumerate}
\end{enumerate}
\end{theorem}

When $q$ is even, every nonzero element of $\Fq$ is positive. Thus a symmetric
matrix in $M_n(\F_q)$ is positive definite if and only if it is strongly nonsingular.
As a consequence, the even-characteristic cases of
Theorem~\ref{thm:positive-main} follow immediately from
Theorem~\ref{thm:strong-main}. In particular, observe that the special form of the two-dimensional preservers in Theorem \ref{thm:positive-main}(1) is characteristic of strong nonsingularity preservers over every finite field. 

We next turn to odd characteristic. Let $\eta$ denote the quadratic character
of $\Fq$, extended by setting $\eta(0)=0$. We call a vector
\[
  \sgnpat=(\sigma_1,\ldots,\sigma_n)\in\{\pm1\}^n
\]
a \emph{sign pattern}. For such a sign pattern $\sgnpat$, let
$\mathcal P_{\sgnpat}(\F_q)$ be the set of symmetric matrices $A\in M_n(\Fq)$
such that
\[
  \eta(\Delta_r(A))=\sigma_r,
  \qquad 1\le r\le n,
\]
where $\Delta_r(A)$ denotes the $r$th leading principal minor. When the underlying field is clear from context, we write $\mathcal P_{\sgnpat}$ instead of $\mathcal P_{\sgnpat}(\F_q)$. 

Matrices with a given sign pattern over the real field were previously considered, for example, in the study of matrix stability \cite{JohnsonMaybeeOleskyVandenDriessche}. Recently, Vishwakarma \cite{Vishwakarma} also developed a Cholesky decomposition theory for symmetric matrices over finite fields. In this setting, the role played over $\mathbb{R}$ by the signs of successive leading principal minors is naturally replaced by their square classes.

We say that
$f$ \emph{preserves $\mathcal P_{\sgnpat}(\F_q)$} if
$A\in\mathcal P_{\sgnpat}(\F_q)$ implies $f[A]\in\mathcal P_{\sgnpat}(\F_q)$. The following result completely classifies such preservers. 

\begin{theorem}\label{thm:sign-main}
Let $q=p^k$ be odd, let $n\ge2$, let
$\sgnpat\in\{\pm1\}^n$, and let $f:\Fq\to\Fq$.
Then $f$ preserves $\mathcal P_{\sgnpat}(\F_q)$ if and only if
\[
  f(x)=cx^{p^j}
\]
for some $c\in\Fqp$ and $0\le j\le k-1$.
\end{theorem}

A notable consequence of Theorem~\ref{thm:sign-main} is that preserving any given sign pattern is equivalent to preserving all sign patterns.

\begin{corollary}\label{cor:all-signs}
Let $q$ be odd and $n\ge2$. If an entrywise map $f:\Fq\to\Fq$ preserves
$\mathcal P_{\sgnpat}(\F_q)$ for one $\sgnpat\in\{\pm1\}^n$, then it preserves
$\mathcal P_{\boldsymbol\tau}(\F_q)$ for every
$\boldsymbol\tau\in\{\pm1\}^n$.
\end{corollary}

The rest of the paper is organized as follows. Section~\ref{sec:prelim} contains the quadratic-character and Lucas-theorem preliminaries. Section~\ref{sec:idempotent} gives the idempotent reduction. Section~\ref{sec:odd} treats the two-dimensional classes in odd characteristic and proves the odd-characteristic part of the classification of positivity preservers in any fixed dimension (Theorem \ref{thm:positive-main}). Section~\ref{sec:strong} classifies the preservers of strong nonsingularity (Theorem~\ref{thm:strong-main}), which also yields the even-characteristic part of Theorem~\ref{thm:positive-main}. Section~\ref{sec:sign} proves Theorem~\ref{thm:sign-main} for arbitrary sign patterns.

\section{Preliminaries}\label{sec:prelim}

Throughout this section, $q=p^k$ is odd. Recall that $\eta$ is the quadratic character of $\Fq$, extended by $\eta(0)=0$. For $\delta\in\{\pm1\}$, set
\[
  C_\delta=\{x\in\Fqx:\eta(x)=\delta\}.
\] Thus $C_{+1}=\Fqp$ and $C_{-1}=\Fqm$. We shall use the standard quadratic-character identity (see, for example, \cite[Theorem~5.48]{LN97})
\begin{equation}\label{eq:char-sum}
  \sum_{t\in\Fq}\eta(t(1-t))=-\eta(-1).
\end{equation}

\begin{lemma}\label{lem:separation}
Let $q$ be an odd prime power.
\begin{enumerate}[label=\textup{(\roman*)}]
\item Assume $q\ge5$. If $x,y\in\Fqx$ are distinct and $\eta(x)=\eta(y)$, then for every
$\delta\in\{\pm1\}$, there exists $z\in\Fqx$ such that
\[
  \eta(z)=-\eta(x),\qquad
  \eta(x-z)=\delta,\qquad
  \eta(y-z)=-\delta.
\]
\item If $x\in\Fqm$ and $y\in\Fqp$, then there exists $u\in\Fqp$ such that $x-u\in\Fqp$ and $y-u\notin\Fqp$.
\end{enumerate}
\end{lemma}

\begin{proof}
(i) First suppose that $x,y\in\Fqp$. For $a\in\Fqp$, let
\[
  E_a=\{z\in\Fqm:a-z\in\Fqp\}.
\]
Since $E_a=aE_1$, all the sets $E_a$ have the same cardinality. Moreover,
\[
\begin{aligned}
4\sum_{z\in E_1}z
&=\sum_{z\in\Fq}z(1-\eta(z))(1+\eta(1-z))=-\sum_{z\in\Fq}z\,\eta\bigl(z(1-z)\bigr),
\end{aligned}
\]
because
\[
  \sum_{z\in\Fq}z
  =\sum_{z\in\Fq}z\eta(z)
  =\sum_{z\in\Fq}z\eta(1-z)=0.
\]
Indeed, the second equality follows from
$z\eta(z)=z^{(q+1)/2}$ and the standard power-sum identity since $(q+1)/2<q-1$, while the
third follows by the substitution $z\mapsto1-z$.

Set
\[
  T=\sum_{z\in\Fq}z\,\eta\bigl(z(1-z)\bigr).
\]
The substitution $z\mapsto1-z$ gives
\[
  T=\sum_{z\in\Fq}(1-z)\eta\bigl(z(1-z)\bigr),
\]
and hence, by identity~\eqref{eq:char-sum},
\[
  2T=\sum_{z\in\Fq}\eta\bigl(z(1-z)\bigr)=-\eta(-1).
\]
Therefore $\sum_{z\in E_a}z=\eta(-1)a/8.$ Thus $E_x\ne E_y$. Since $|E_x|=|E_y|$, both
$E_x\setminus E_y$ and $E_y\setminus E_x$ are nonempty. Choose
$z\in E_x\setminus E_y$ when $\delta=1$, and
$z\in E_y\setminus E_x$ when $\delta=-1$. This proves \textup{(i)}
when $x,y\in\Fqp$.

Now suppose that $x,y\in\Fqm$. Choose $\nu\in\Fqm$, so that
$\nu x,\nu y\in\Fqp$. Using the above argument, there exists $w\in\Fqm$ such that $\eta(\nu x-w)=-\delta$ and $ \eta(\nu y-w)=\delta.$
Thus $z=w/\nu\in\Fqp$ satisfies $\eta(x-z)=\delta$ and $\eta(y-z)=-\delta$, as required.

(ii) For each $t\in\Fqx$, set $  N(t)=\#\{u\in\Fqp:t-u\in\Fqp\}.$ 
By identity~\eqref{eq:char-sum}, we have
\[
\begin{aligned}
4N(t)
&=\sum_{u\ne0,t}\big(1+\eta(u)\big)\big(1+\eta(t-u)\big)=q-2-2\eta(t)-\eta(-1).
\end{aligned}
\]
Therefore, if $x\in\Fqm$ and $y\in\Fqp$, then $N(x)=N(y)+1.$ Hence some $u\in\Fqp$ satisfies $x-u\in\Fqp$ and
$y-u\notin\Fqp$, as required.
\end{proof}

We also record the following consequence of Lucas's theorem \cite{Lucas}. The lemma is implicit in \cite[Section 7]{GGVY}. Here we include a short proof for completeness.

\begin{lemma}\label{lem:Lucas}
Let $q=p^k$ be odd, put $d=(q-1)/2$, and let $1\le m\le q-1$ satisfy $\gcd(m,q-1)=1$. Then
\begin{equation}\label{eq:Lucas}
  (x^m-1)^d=(x-1)^d, \qquad \forall x\in \F_q
\end{equation}
holds if and only if $m=p^j$ for some $0\le j\le k-1$.
\end{lemma}

\begin{proof}
First assume $m=p^j$ for some $0\le j\le k-1$. Since $m$ is odd, $d(m-1)$ is divisible by $q-1$. Hence
\[
  (x^m-1)^d=(x-1)^{dm}=(x-1)^d
\]
for every $x\in\Fq$.

Conversely, assume \eqref{eq:Lucas}, and write
\(d=\ell(1+p+\cdots+p^{k-1})\), where \(\ell=(p-1)/2\). For \(a>0\),
let \(\rho(a)\in\{1,\ldots,q-1\}\) denote its least positive residue
modulo \(q-1\).

Expanding \((X^m-1)^d\) and reducing modulo \(X^q-X\), we see that
\(\rho(mr)\le d\) whenever \(\binom{d}{r}\not\equiv0\pmod p\).
Indeed, the residues \(\rho(mr)\), \(1\le r\le d\), are distinct since
\(\gcd(m,q-1)=1\), whereas \((X-1)^d\) has degree \(d\). Since all
base-\(p\) digits of \(d\) equal \(\ell\), Lucas's theorem gives
\begin{equation}\label{eq:residue-bound}
  \rho(mr)\le d
\end{equation}
whenever all base-\(p\) digits of \(r\) are at most \(\ell\).

Write \(m=\sum_{i=0}^{k-1}m_ip^i\). For each \(j\), since
\(p^k\equiv1\pmod{q-1}\),
\[
  mp^{k-1-j}\equiv
  \sum_{i=0}^{j}m_i p^{k-1-j+i}
  +\sum_{i=j+1}^{k-1}m_i p^{i-j-1}
  \pmod{q-1}.
\]
The integer on the right lies in \([1,q-1]\) and is not \(q-1\), since
\(\gcd(m,q-1)=1\). It is therefore \(\rho(mp^{k-1-j})\), and its
leading base-\(p\) digit is \(m_j\). Taking
\(r=p^{k-1-j}\) in \eqref{eq:residue-bound} gives \(m_j\le\ell\).
Thus \(m_i\le\ell\) for every \(i\).

Let \(t=\max_i m_i\). If \(t>1\), choose \(j\) with \(m_j=t\), put
\(u=\lfloor\ell/t\rfloor+1\), and take \(r=up^{k-1-j}\). Then
\(1\le u\le\ell\), \(ut>\ell\), and
\(um_i\le ut\le\ell+t\le2\ell=p-1\). Hence multiplying the displayed
base-\(p\) expansion by \(u\) produces no carries. Its leading digit
becomes \(ut>\ell\), so \(\rho(mr)>d\), contradicting
\eqref{eq:residue-bound}. Therefore \(m_i\in\{0,1\}\) for every \(i\).

Finally, suppose that \(m_j=m_\nu=1\) for distinct \(j,\nu\), and take
\(r=\ell(p^{k-1-j}+p^{k-1-\nu})\). The two corresponding base-\(p\)
expansions above add without carry, since every resulting digit is at
most \(2\ell=p-1\), while their leading digits add to
\(2\ell=p-1>\ell\). Since \(mr\not\equiv0\pmod{q-1}\), it follows that
\(\rho(mr)>d\), again contradicting \eqref{eq:residue-bound}.

Thus \(m\) has exactly one nonzero base-\(p\) digit, equal to \(1\).
Hence \(m=p^j\) for some \(0\le j\le k-1\).
\end{proof}

\section{Idempotent reduction}\label{sec:idempotent}

For a function $f:X\to X$ and $m\ge0$, write $f^{(m)}$ for its $m$-fold iterate, with $f^{(0)}=\operatorname{id}$.

\begin{lemma}\label{lem:finite-iteration}
Let $X$ be finite and let $\mathcal S$ be a family of functions $X\to X$ that is closed under composition. If the identity is the only idempotent function in $\mathcal S$, then every $f\in\mathcal S$ is a permutation and $f^{-1}\in\mathcal S$.
\end{lemma}

\begin{proof}
Let $f\in \mathcal S$. Since there are only finitely many functions $X\to X$, we can choose $1\le r<s$ with $f^{(r)}=f^{(s)}$. Then the iterates are periodic from $r$ onward with period $s-r$. Choose $m\ge r$ divisible by $s-r$. Thus $f^{(2m)}=f^{(m)}$, so $f^{(m)}$ is idempotent. Hence $f^{(m)}=\operatorname{id}$ and $f^{-1}=f^{(m-1)}\in\mathcal S$.
\end{proof}

\begin{lemma}\label{lem:idempotent-strong}
Let $q$ be a prime power. If $h:\Fq\to\Fq$ is idempotent and preserves strong nonsingularity on symmetric matrices in $M_2(\F_q)$, then $h=\operatorname{id}$.
\end{lemma}

\begin{proof}
For each $a\in \Fqx$, preservation of $\operatorname{diag}(a,1)$ gives $h(a)\ne0$. For distinct $a,b\in\Fqx$,
\[
  \det\begin{pmatrix}a&a\\a&b\end{pmatrix}=a(b-a)\ne0,
\]
so preservation gives $h(a)\ne h(b)$. Thus $h$ is injective on $\Fqx$. For each $a\in \Fqx$, since both $a$ and $h(a)$ lie in $\Fqx$ and $h(h(a))=h(a)$, injectivity gives $h(a)=a$. If $h(0)=c\ne0$, then $h(c)=c$, while $cI_2$ is sent to the singular constant matrix with entry $c$, a contradiction. Thus, $h=\operatorname{id}$.
\end{proof}

For odd $q$ and $\delta\in\{\pm1\}$, recall that
$\mathcal P_{(1,\delta)}$ consists of the symmetric matrices 
\begin{equation}\label{EPdelta}
  A=\begin{pmatrix}a&b\\b&c\end{pmatrix}
  \quad\text{such that}\quad
  a\in\Fqp,\qquad ac-b^2\in C_\delta.
\end{equation}
Thus $\mathcal P_{(1,1)}$ is precisely the set of positive definite
$2\times2$ matrices, while $\mathcal P_{(1,-1)}$ corresponds to the
sign pattern $(1,-1)$.

\begin{proposition}\label{prop:idempotent-class}
Let $q$ be odd and $\delta\in\{\pm1\}$. If $h:\Fq\to\Fq$ is idempotent and preserves $\mathcal P_{(1,\delta)}$, then $h=\operatorname{id}$.
\end{proposition}
\begin{proof}
We show successively that $h$ fixes $\Fqp$, $0$, and $\Fqm$.

Clearly $h(\Fqp)\subseteq\Fqp$. We first show that $h$ is injective on $\Fqp$.
If $-1\in\Fqm$, then for distinct $a,b\in\Fqp$,
$\eta(a-b)=-\eta(b-a)$. After interchanging $a,b$ if necessary,
\[
  \begin{pmatrix}a&a\\a&b\end{pmatrix}
  \in\mathcal P_{(1,\delta)},
\]
while its image would be singular if $h(a)=h(b)$. Hence $h(a)\ne h(b)$.

Now suppose $-1\in\Fqp$ and $h(a)=h(b)=s$ for distinct $a,b\in\Fqp$.
Then $q\ge5$, so Lemma~\ref{lem:separation}\textup{(i)} gives $z\in\Fqm$ such that
\[
  \eta(a-z)=\delta,
  \qquad
  \eta(b-z)=-\delta.
\]
Put $v=h(z)$. Since $-1\in \F_q^+$, both
\[
  \begin{pmatrix}a&a\\a&z\end{pmatrix},
  \qquad
  \begin{pmatrix}b&z\\z&z\end{pmatrix}
\]
belong to $\mathcal P_{(1,\delta)}$. Their images have determinants
$s(v-s)$ and $v(s-v)$ in $C_\delta$. Since $s$ and $-1$ are both in $\F_q^+$, the
first condition gives $s-v\in C_\delta$, and the second then gives
$v\in\Fqp$.
By idempotence, $h(v)=v$. Since $v\in\Fqp$, $z\in\Fqm$, and $-1\in\Fqp$, the determinants of
\[
  \begin{pmatrix}v&v\\v&z\end{pmatrix},
  \qquad
  \begin{pmatrix}v&z\\z&z\end{pmatrix}
\]
satisfy
$\eta\bigl(v(z-v)\bigr)=-\eta\bigl(z(v-z)\bigr)$. Thus one of these
matrices lies in $\mathcal P_{(1,\delta)}$, while its image is singular,
a contradiction.
Hence $h$ is injective on $\Fqp$, and idempotence gives
\begin{equation}\label{eq:h-fixes-squares}
  h(u)=u,\qquad \forall u\in\Fqp.
\end{equation}

We next show that $h(0)=0$. Put $c=h(0)$ and suppose $c\ne0$. If
$\delta=1$, then $\operatorname{diag}(c^2,1)\in\mathcal P_{(1,1)}$,
while its image
\[
  \begin{pmatrix}c^2&c\\c&1\end{pmatrix}
\]
is singular, a contradiction. Next suppose $\delta=-1$. Fix $z\in\Fqm$ and put $v=h(z)$. For every $u\in\Fqp$, we have $\operatorname{diag}(u,z)\in\mathcal P_{(1,-1)}$; by~\eqref{eq:h-fixes-squares}, its image is
\[
\begin{pmatrix}u&c\\c&v\end{pmatrix}\in\mathcal P_{(1,-1)}.
\]
Thus, 
\begin{equation}\label{eq:mixed-h0}
  uv-c^2\in\Fqm, \qquad \forall u\in\Fqp.
\end{equation}
If $v=0$, idempotence gives $c=h(0)=h(v)=v=0$, a contradiction. If
$v\in\Fqp$, taking $u=c^2/v$ makes $uv-c^2=0$, again a contradiction.
Hence $v\in\Fqm$. By Lemma~\ref{lem:separation}\textup{(ii)}, choose
$w\in\Fqp$ such that $v/c^2-w\in\Fqp$. Taking $u=w^{-1}$ in
\eqref{eq:mixed-h0} then gives
\[
  uv-c^2=\frac{c^2}{w}\left(\frac{v}{c^2}-w\right)\in\Fqp,
\]
again a contradiction. Thus $h(0)=0$.

It remains to show that $h$ fixes $\Fqm$. Let $z\in\Fqm$. If
$u\in\Fqp$ and $z-u\in C_\delta$, then
\[
  \begin{pmatrix}u&u\\u&z\end{pmatrix}
  \in\mathcal P_{(1,\delta)};
\]
by \eqref{eq:h-fixes-squares}, its image
\[
  \begin{pmatrix}u&u\\u&h(z)\end{pmatrix}
\]
also lies in $\mathcal P_{(1,\delta)}$. Thus, for each $u\in\Fqp$,
\begin{equation}\label{eq:h-difference}
  z-u\in C_\delta
  \quad\Longrightarrow\quad
  h(z)-u\in C_\delta.
\end{equation}

We first show that $h(z)\in\Fqm$. If $\delta=-1$, this follows from
$\operatorname{diag}(1,z)\in\mathcal P_{(1,-1)}$. Next suppose $\delta=1$.
If $h(z)\in\Fqp$, Lemma~\ref{lem:separation}\textup{(ii)}, applied to
$z$ and $h(z)$, gives $u\in\Fqp$ with $z-u\in\Fqp$ and $h(z)-u\notin\Fqp$, contradicting \eqref{eq:h-difference}. Thus it remains only to exclude
$h(z)=0$. If $-1\in\Fqm$, Lemma~\ref{lem:separation}\textup{(ii)}, applied
to $z$ and $1$, gives $u\in\Fqp$ with $z-u\in\Fqp$, whereas
$h(z)-u=-u\in\Fqm$, again contradicting \eqref{eq:h-difference}. If
$-1\in\Fqp$, the matrix
\[
  \begin{pmatrix}1&z\\z&0\end{pmatrix}
\]
lies in $\mathcal P_{(1,1)}$, but its image would be singular. Hence
$h(z)\in\Fqm$.

Finally, suppose $h(z)\ne z$. Then necessarily $q\ge5$, and
Lemma~\ref{lem:separation}\textup{(i)}, applied to $z$ and $h(z)$, gives $u\in\Fqp$ such that $\eta(z-u)=\delta$ and $\eta(h(z)-u)=-\delta$, contradicting \eqref{eq:h-difference}. Therefore $h(z)=z$ for every
$z\in\Fqm$. We conclude that $h=\operatorname{id}$.
\end{proof}

The preceding results immediately give the following corollary.
\begin{corollary}\label{cor:inverse-preservation}
The following hold.
\begin{enumerate}[label=\textup{(\roman*)}]
\item If $q$ is odd and $\delta\in\{\pm1\}$, then every preserver of
$\mathcal P_{(1,\delta)}$ is a permutation whose inverse also preserves
$\mathcal P_{(1,\delta)}$.
\item For any prime power $q$ and $n\ge2$, every map preserving strong
nonsingularity on symmetric matrices in $M_n(\Fq)$ is a permutation
whose inverse has the same property.
\end{enumerate}
\end{corollary}

\begin{proof}
In both cases the relevant family of maps is closed under composition.
Part \textup{(i)} follows from Proposition~\ref{prop:idempotent-class}
and Lemma~\ref{lem:finite-iteration}. For \textup{(ii)}, an idempotent
preserver restricts to an idempotent preserver of strong nonsingularity
on symmetric $2\times2$ matrices by adjoining an identity block.
Hence Lemma~\ref{lem:idempotent-strong} and
Lemma~\ref{lem:finite-iteration} give the result.
\end{proof}

In particular, by part \textup{(i)}, if $q$ is odd and $f$ preserves
$\mathcal P_{(1,\delta)}$, then for every symmetric $A\in M_2(\Fq)$,
\begin{equation}\label{eq:class-equivalence}
  A\in\mathcal P_{(1,\delta)}
  \quad\Longleftrightarrow\quad
  f[A]\in\mathcal P_{(1,\delta)}.
\end{equation}

\section{Odd characteristic}\label{sec:odd}
Throughout this section, $q=p^k$ is odd and $\delta\in\{\pm1\}$.

\begin{proposition}\label{prop:class-multiplicative}
If $f:\Fq\to\Fq$ preserves $\mathcal P_{(1,\delta)}$ and $f(1)=1$, then
$f$ is multiplicative.
\end{proposition}

\begin{proof}
By Corollary~\ref{cor:inverse-preservation}\textup{(i)}, $f$ is a
permutation and \eqref{eq:class-equivalence} holds. If $q=3$, the only
nontrivial permutation fixing $1$ interchanges $0$ and $2$. This fails
to preserve $\operatorname{diag}(1,1)$ when $\delta=1$, and
$\left(\begin{smallmatrix}1&1\\1&0\end{smallmatrix}\right)$ when
$\delta=-1$. Hence $f=\operatorname{id}$. We may therefore assume
$q>3$.

Since every square occurs as the first leading principal minor of a
matrix in $\mathcal P_{(1,\delta)}$ and $f$ is a permutation, we have $f(\Fqp)=\Fqp$. We shall
also use the following elementary observation:
\begin{equation}\label{eq:translate-unique}
  C_\delta+t=C_\delta \quad\Longrightarrow\quad t=0.
\end{equation}
Indeed, if $t\ne0$, then translation by $t$ partitions $\Fq$ into
orbits of size $p$. Hence any subset invariant under this translation
has cardinality divisible by $p$, whereas
$|C_\delta|=(q-1)/2$ is not divisible by $p$.

We first show that $f(0)=0$. Put $r=f(0)^2$. For $a\in\Fqp$ and
$c\in\Fq$, applying \eqref{eq:class-equivalence} to
$\operatorname{diag}(a,c)$ gives
\[
  c\in C_\delta
  \quad\Longleftrightarrow\quad
  f(c)-\frac{r}{f(a)}\in C_\delta.
\]
Hence $f(C_\delta)=r/f(a)+C_\delta$. If $r\ne0$, then choosing distinct
$a_1,a_2\in\Fqp$ gives two distinct translates of $C_\delta$,
contradicting \eqref{eq:translate-unique}. Thus $f(0)=0$. Since $f$ is
a permutation and $f(\Fqp)=\Fqp$, it follows that $f(\Fqm)=\Fqm$. We have therefore shown
\begin{equation}\label{eq:class-images}
  f(0)=0,\qquad f(\Fqp)=\Fqp,\qquad f(\Fqm)=\Fqm.
\end{equation}

We next derive a useful identity. Fix $a\in\Fqp$ and $b\in\Fq$, and set
$\alpha=b^2/a$. Applying \eqref{eq:class-equivalence} to
$\left(\begin{smallmatrix}a&b\\b&c\end{smallmatrix}\right)$ gives
\[
  c-\alpha\in C_\delta
  \quad\Longleftrightarrow\quad
  f(c)-\frac{f(b)^2}{f(a)}\in C_\delta,
\]
and hence
\[
  f(\alpha+C_\delta)
  =\frac{f(b)^2}{f(a)}+C_\delta.
\]
If $b\ne0$, then $\alpha\in\Fqp$, and applying the same argument with
$a=b=\alpha$ gives
$f(\alpha+C_\delta)=f(\alpha)+C_\delta$. Thus
\eqref{eq:translate-unique} yields
$f(\alpha)=f(b)^2/f(a)$. When $b=0$, the same conclusion follows from $f(0)=0$. Therefore
\[
  f\!\left(\frac{b^2}{a}\right)
  =\frac{f(b)^2}{f(a)},
  \qquad \forall a\in\Fqp,\ b\in\Fq.
\]
Taking $a=1$ gives
\begin{equation}\label{eq:squares}
  f(b^2)=f(b)^2,
  \qquad \forall b\in\Fq.
\end{equation}
Hence, if $x,a\in\Fqp$ and $b^2=x$, then $f(x/a)=f(x)/f(a).$ It follows that
\begin{equation}\label{eq:mult-squares}
  f(st)=f(s)f(t),
  \qquad \forall s,t\in\Fqp.
\end{equation}

It remains to extend multiplicativity to all of $\Fq$. Let $s\in\Fqp$ and $z\in\Fqm$, and set
\[
  y=\frac{f(sz)}{f(s)}\in\Fqm.
\]
Suppose that $y\ne f(z)$. By Lemma~\ref{lem:separation}\textup{(i)},
there exists $v\in\Fqp$ such that
\[
  y-v\in C_\delta,
  \qquad
  f(z)-v\in C_{-\delta}.
\]
By \eqref{eq:class-images}, we can write $v=f(u)$ for some $u\in\Fqp$.
By \eqref{eq:mult-squares},
\[
  f(sz)-f(su)=f(s)(y-v)\in C_\delta.
\]
Hence
\[
  f\!\left[
    \begin{pmatrix}
      su&su\\
      su&sz
    \end{pmatrix}
  \right]
  \in\mathcal P_{(1,\delta)}.
\]
By \eqref{eq:class-equivalence}, \[\begin{pmatrix}
      su&su\\
      su&sz
    \end{pmatrix} \in \mathcal P_{(1,\delta)}.\] 
Its determinant is $s^2u(z-u)$. Since $s^2u\in\Fqp$, we obtain
$z-u\in C_\delta$. It follows that
\[
  \begin{pmatrix}
    u&u\\
    u&z
  \end{pmatrix}
  \in\mathcal P_{(1,\delta)}.
\]
Applying $f$ to this matrix gives $f(z)-v\in C_\delta$, contradicting
$f(z)-v\in C_{-\delta}$. Hence
\begin{equation}\label{eq:mult-mixed}
  f(sz)=f(s)f(z),
  \qquad \forall s\in\Fqp,\ z\in\Fqm.
\end{equation}

Finally, let $z,w\in\Fqm$. Writing $w=zs$ with
$s=w/z\in\Fqp$, \eqref{eq:mult-squares}, \eqref{eq:squares}, and
\eqref{eq:mult-mixed} give
\[
  f(zw)=f(z^2s)=f(z^2)f(s)=f(z)^2f(s)=f(z)f(w).
\]
Together with $f(0)=0$, this proves that $f$ is multiplicative on
$\Fq$.
\end{proof}

\begin{theorem}\label{thm:canonical-classes}
A function $f:\Fq\to\Fq$ preserves $\mathcal P_{(1,\delta)}$ if and only if
\[
  f(x)=cx^{p^j}
\]
for some $c\in\Fqp$ and $0\le j\le k-1$.
\end{theorem}

\begin{proof}
Let $f$ preserve $\mathcal P_{(1,\delta)}$. Put $c=f(1)\in\Fqp$ and
$g=c^{-1}f$. Since multiplying every entry by $c^{-1}$ multiplies the
first leading principal minor by $c^{-1}$ and the determinant by
$c^{-2}$, the map $g$ again preserves $\mathcal P_{(1,\delta)}$ and
$g(1)=1$. By Proposition~\ref{prop:class-multiplicative}, $g$ is a
multiplicative permutation, so $g(x)=x^m$ for some $1\le m\le q-1$ with $\gcd(m,q-1)=1$.

By Corollary~\ref{cor:inverse-preservation}\textup{(i)}, both $g$ and $g^{-1}$ preserve $\mathcal P_{(1,\delta)}$. Hence, by considering 
\[
\begin{pmatrix}
      1&1\\
      1&x
    \end{pmatrix}
\]
for $x\in \F_q$, we obtain that
\[
  x-1\in C_\delta
  \quad\Longleftrightarrow\quad
  x^m-1\in C_\delta.
\]
Since $x^m=1$ if and only if $x=1$, this implies
$\eta(x^m-1)=\eta(x-1)$ for every $x\in\Fq$. As
$y^{(q-1)/2}=\eta(y)$ for all $y\in\Fq$, this is exactly
\eqref{eq:Lucas}. Hence Lemma~\ref{lem:Lucas} gives $m=p^j$ for some $0\le j\le k-1$.

Conversely, Frobenius automorphisms and multiplication by a nonzero square
preserve the signs of both the first leading principal minor and the determinant.
\end{proof}

\begin{corollary}\label{cor:odd-positive}
Let $q=p^k$ be odd and $n\ge2$. A function $f:\Fq\to\Fq$ preserves positivity on $M_n(\Fq)$ if and only if
\[
  f(x)=cx^{p^j}
\]
for some $c\in\Fqp$ and $0\le j\le k-1$.
\end{corollary}

\begin{proof}
A positivity preserver on $M_n(\Fq)$ restricts to one on $M_2(\Fq)$ by extending a positive definite $2\times2$ matrix with an identity block. The result therefore follows from Theorem~\ref{thm:canonical-classes}. Conversely, a map $x\mapsto cx^{p^j}$ with $c\in\Fqp$ sends every leading principal minor $\Delta_r$ to $c^r\Delta_r^{p^j}$, so it preserves positivity.
\end{proof}

\section{Strong nonsingularity}\label{sec:strong}

We now prove Theorem~\ref{thm:strong-main}.

\begin{proof}[Proof of Theorem~\ref{thm:strong-main}]
Clearly \textup{(i)} implies \textup{(ii)}. Assume \textup{(ii)}. By
Corollary~\ref{cor:inverse-preservation}\textup{(ii)}, $f$ is a permutation and
$f^{-1}$ has the same preservation property. Adjoining identity blocks
shows that both $f$ and $f^{-1}$ preserve strong nonsingularity on symmetric
$r\times r$ matrices for every $2\le r\le n$. In particular, $f(a)\ne0$ for $a\ne0$, by applying this
to $\operatorname{diag}(a,1)$; hence $f(0)=0$. Put
\[
  c=f(1)\in\Fqx,
  \qquad
  g=c^{-1}f.
\]
Scaling every entry by $c^{-1}$ preserves strong nonsingularity, so $g$
and $g^{-1}$ have the same preservation properties, with $g(0)=0$ and
$g(1)=1$.

If $A\in M_r(\Fq)$ is symmetric, $2\le r\le n$, and its first $r-1$
leading principal minors are nonzero, then
\begin{equation}\label{eq:boundary-equivalence}
  \det A=0
  \quad\Longleftrightarrow\quad
  \det g[A]=0.
\end{equation}
Indeed, the first $r-1$ leading principal minors of $g[A]$ are nonzero:
for $r=2$ this follows from $g(\Fq^\times)=\Fq^\times$, while for $r\ge3$ it
follows by applying the preservation property in dimension $r-1$. Hence $A$
is strongly nonsingular if and only if $\det A\ne0$, and similarly for
$g[A]$. Applying the preservation property to both $g$ and $g^{-1}$ gives \eqref{eq:boundary-equivalence}.

For $a\in\Fqx$ and $b\in\Fq$, apply \eqref{eq:boundary-equivalence} to
\[
  A_{a,b}=\begin{pmatrix}a&b\\ b&b^2/a\end{pmatrix}.
\]
Since $\det A_{a,b}=0$, we obtain
\begin{equation}\label{eq:strong-boundary}
  g(a)g\!\left(\frac{b^2}{a}\right)=g(b)^2.
\end{equation}

If $q=2$, then $g=\operatorname{id}$ and the conclusion is immediate.
Assume $q>2$, put $N=q-1$, and let $\omega$ generate $\Fqx$. Since $g$
permutes $\Fqx$ and fixes $1$, write
$g(\omega^r)=\omega^{\psi(r)}$, where $\psi$ is a permutation of
$\mathbb Z/N\mathbb Z$ with $\psi(0)=0$. Taking $a=\omega^r$ and
$b=\omega^{r+1}$ in \eqref{eq:strong-boundary} gives
\[
  \psi(r)+\psi(r+2)=2\psi(r+1)\pmod N.
\]
Thus $\psi(r+1)-\psi(r)$ is constant in $\mathbb Z/N\mathbb Z$.
Since $\psi(0)=0$, we have $\psi(r)=mr$ for some
$m\in\mathbb Z/N\mathbb Z$. As $\psi$ is a permutation, $m$ is a unit
modulo $N$. Choosing the representative $1\le m\le N-1$, we obtain
$g(x)=x^m$ and $\gcd(m,N)=1$. This proves \textup{(iii)} when $n=2$.

Now suppose $n\ge3$. The two-dimensional case gives $g(x)=x^m$, so $g$
is multiplicative.
For $t\in\Fqx$, let
\[
  A_t=
  \begin{pmatrix}
    1&1&0\\
    1&1+t&t\\
    0&t&t
  \end{pmatrix}.
\]
Its first two leading principal minors are $1$ and $t$, while
$\det A_t=0$. Hence \eqref{eq:boundary-equivalence} gives
\[
  0=\det g[A_t]
   =g(t)\bigl(g(1+t)-1-g(t)\bigr),
\]
so $g(1+t)=1+g(t)$ for $t\in\Fq$, with $t=0$ immediate. Therefore, for
$y\ne0$, by multiplicativity,
\[
  g(x+y)=g(y)g(x/y+1)=g(y)(1+g(x/y))=g(x)+g(y).
\]
Thus $g$ is a field automorphism, and hence $g(x)=x^{p^j}$ for some
$0\le j\le k-1$.

It remains to prove that \textup{(iii)} implies \textup{(i)}. If $n=2$
and $f(x)=cx^m$, then for any $2\times2$ matrix
\[
  \det f\!\left[\begin{pmatrix}a&b\\ d&e\end{pmatrix}\right]
  =c^2\bigl((ae)^m-(bd)^m\bigr).
\]
Since $x\mapsto x^m$ permutes $\Fq$, both leading principal minors are
nonzero whenever the original ones are. If $n\ge3$ and $B$ is an
$r\times r$ leading principal submatrix, then
$\det f[B]=c^r(\det B)^{p^j}$, so again every nonzero leading principal
minor remains nonzero. Thus the maps in \textup{(iii)} preserve strong
nonsingularity on all matrices.
\end{proof}

When $q=2^k$, every nonzero element of $\Fq$ is a square. Thus a symmetric matrix is positive definite if and only if it is strongly nonsingular. Consequently, the even-characteristic cases of Theorem~\ref{thm:positive-main} follow immediately from Theorem~\ref{thm:strong-main}.

\begin{proof}[Proof of Theorem~\ref{thm:positive-main}]
For odd $q$, the result is Corollary~\ref{cor:odd-positive}. For even $q$, it follows from Theorem~\ref{thm:strong-main} and the preceding observation.
\end{proof}

\section{Prescribed sign patterns of leading principal minors}\label{sec:sign}
Assume throughout this section that $q=p^k$ is odd.

\begin{proof}[Proof of Theorem~\ref{thm:sign-main}]
Suppose $f$ preserves $\mathcal P_{\sgnpat}(\Fq)$, where $\sgnpat=(\sigma_1,\ldots,\sigma_n)$. Choose $\nu\in\Fqx$ with $\eta(\nu)=\sigma_1$, and define
\[
  g(x)=\nu^{-1}f(\nu x).
\]

We claim that $g$ preserves $\mathcal P_{(1,\sigma_2)}$. Let $B\in\mathcal P_{(1,\sigma_2)}$. For $3\le r\le n$, choose $d_r\in\Fqx$ with $\eta(d_r)=\sigma_{r-1}\sigma_r$, and set
\[
  A=\nu B\oplus\operatorname{diag}(d_3,\ldots,d_n),
\]
with the diagonal block omitted when $n=2$. The first two leading principal minors of $A$ have signs $\sigma_1,\sigma_2$, and for $r\ge3$,
\[
  \Delta_r(A)=\Delta_{r-1}(A)d_r.
\]
Hence $A\in\mathcal P_{\sgnpat}$, so $f[A]\in\mathcal P_{\sgnpat}$. The leading $2\times2$ block of $f[A]$ is $f[\nu B]$, while
$g[B]=\nu^{-1}f[\nu B]$. Thus the first leading principal minor of
$g[B]$ has sign
\[
  \eta(\nu^{-1})\sigma_1=1,
\]
whereas its determinant has sign $\sigma_2$, since it is multiplied by
the square $\nu^{-2}$. Therefore
$g[B]\in\mathcal P_{(1,\sigma_2)}$.

By Theorem~\ref{thm:canonical-classes},
$g(x)=dx^{p^j}$ for some $d\in\Fqp$ and $0\le j\le k-1$. Hence
\[
  f(y)=d\nu^{1-p^j}y^{p^j}.
\]
Since $1-p^j$ is even, $d\nu^{1-p^j}\in\Fqp$.

Conversely, if $f(x)=cx^{p^j}$ with $c\in\Fqp$, then for every symmetric
$A$ and $1\le r\le n$,
\[
  \Delta_r(f[A])=c^r\Delta_r(A)^{p^j}.
\]
Hence $f$ preserves $\mathcal P_{\boldsymbol\tau}$ for every
$\boldsymbol\tau\in\{\pm1\}^n$.
\end{proof}

\begin{remark}\label{rem:LDL}
If a symmetric matrix has all leading principal minors nonzero, then it admits a decomposition
\[
  A=LDL^T,
\]
where $L$ is unit lower triangular and $D=\operatorname{diag}(d_1,\ldots,d_n)$, with
\[
  d_r=\frac{\Delta_r(A)}{\Delta_{r-1}(A)},
  \qquad \Delta_0(A):=1.
\]
Thus $A\in\mathcal P_{\sgnpat}$ exactly when
\[
  \eta(d_r)=\sigma_{r-1}\sigma_r,
  \qquad \sigma_0:=1.
\]
As a consequence, the sign pattern of the leading principal minors is equivalently encoded by the signs of the diagonal pivots in the $LDL^T$ factorization.
\end{remark}

\section*{Acknowledgments}

\noindent{\bf AI disclosure statement.} The authors acknowledge the use of ChatGPT by OpenAI to assist in exploring and refining some of the arguments presented in this paper. All mathematical arguments and results were independently verified by the authors, who take full responsibility for the content of the paper.
\medskip

Part of this work
was completed during a BIRS Research in Teams program at University of British Columbia Okanagan in August 2025. The authors would like to
express their gratitude to the Banff International Research Station for its support and to UBC Okanagan for
its hospitality. D.G. was partially supported by NSF grant \#2350067. H.G. acknowledges
support from PIMS (Pacific Institute for the Mathematical Sciences)
Postdoctoral Fellowships. P.K.V. was supported by the Centre de recherches
math\'ematiques and Universit\'e Laval (CRM--Laval) Postdoctoral Fellowship,
and he acknowledges support from a SwarnaJayanti Fellowship from DST and
SERB (Govt.~of India). C.H.Y. thanks Dragos Ghioca for helpful discussions.


\begin{thebibliography}{99}

\bibitem{AyyerPrasad}
A.~Ayyer and S.~Prasad,
\emph{Positive definite, positive semidefinite and totally positive matrices over finite fields},
arXiv:2608.17702, 2026.

\bibitem{BGKP1}
A.~Belton, D.~Guillot, A.~Khare, and M.~Putinar,
\emph{A panorama of positivity. I: Dimension free},
in \emph{Analysis of Operators on Function Spaces: The Serguei Shimorin Memorial Volume},
Birkh\"auser, Cham, 2019, pp.~117--165.

\bibitem{BGKP2}
A.~Belton, D.~Guillot, A.~Khare, and M.~Putinar,
\emph{A panorama of positivity. II: Fixed dimension},
in \emph{Complex Analysis and Spectral Theory}, Contemp. Math. 743,
Amer. Math. Soc., Providence, RI, 2020, pp.~109--150.

\bibitem{Blokhuis}
A.~Blokhuis,
\emph{On subsets of $\operatorname{GF}(q^2)$ with square differences},
Nederl. Akad. Wetensch. Indag. Math. 46 (1984), no.~4, 369--372.

\bibitem{Carlitz}
L.~Carlitz,
\emph{A theorem on permutations in a finite field},
Proc. Amer. Math. Soc. 11 (1960), 456--459.

\bibitem{CooperHannaWhitlatch}
J.~Cooper, E.~Hanna, and H.~Whitlatch,
\emph{Positive-definite matrices over finite fields},
Rocky Mountain J. Math. 54 (2024), 423--438.

\bibitem{CrootLev}
E.~S.~Croot III and V.~F.~Lev,
\emph{Open problems in additive combinatorics},
in \emph{Additive Combinatorics},
CRM Proc. Lecture Notes 43,
Amer. Math. Soc., Providence, RI, 2007, pp.~207--233.

\bibitem{GolubVanLoan}
G.~H.~Golub and C.~F.~Van Loan,
\emph{Matrix Computations},
4th ed., Johns Hopkins University Press, Baltimore, 2013.

\bibitem{GoryainovEtAl}
S.~V.~Goryainov, V.~V.~Kabanov, L.~V.~Shalaginov, and A.~A.~Valyuzhenich,
\emph{On eigenfunctions and maximal cliques of Paley graphs of square order},
Finite Fields Appl. 52 (2018), 361--369.

\bibitem{GGVY}
D.~Guillot, H.~Gupta, P.~K. Vishwakarma, and C.~H. Yip,
\emph{Positivity preservers over finite fields},
J. Algebra 684 (2025), 479--523.

\bibitem{GGVYSign}
D.~Guillot, H.~Gupta, P.~K. Vishwakarma, and C.~H. Yip,
\emph{Entrywise transforms preserving matrix positivity and nonpositivity},
J. London Math. Soc. 113 (2026), no.~5, e70560.

\bibitem{JohnsonMaybeeOleskyVandenDriessche}
C.~R.~Johnson, J.~S.~Maybee, D.~D.~Olesky, and P.~van den Driessche,
\emph{Nested sequences of principal minors and potential stability},
Linear Algebra Appl. \textbf{262} (1997), 243--257.

\bibitem{KhareBook}
A.~Khare,
\emph{Matrix Analysis and Entrywise Positivity Preservers},
London Math. Soc. Lecture Note Ser. 471,
Cambridge Univ. Press, Cambridge, 2022.

\bibitem{LN97}
R.~Lidl and H.~Niederreiter,
\emph{Finite Fields}, second ed.,
Encyclopedia of Mathematics and its Applications, vol.~20,
Cambridge Univ. Press, Cambridge, 1997.

\bibitem{Lucas}
E.~Lucas,
\emph{Th\'eorie des fonctions num\'eriques simplement p\'eriodiques},
Amer. J. Math. 1 (1878), 289--321.

\bibitem{MuzychukKovacs}
M.~Muzychuk and I.~Kov\'acs,
\emph{A solution of a problem of A. E. Brouwer},
Des. Codes Cryptogr. 34 (2005), 249--264.

\bibitem{Schoenberg1942}
I.~J.~Schoenberg,
\emph{Positive definite functions on spheres},
Duke Math. J. 9 (1942), 96--108.

\bibitem{VandebrilVanBarelMastronardi}
R.~Vandebril, M.~Van Barel, and N.~Mastronardi,
\emph{Matrix Computations and Semiseparable Matrices, Vol.~1: Linear Systems},
Johns Hopkins University Press, Baltimore, 2008.

\bibitem{Vishwakarma}
P.~K.~Vishwakarma,
\emph{Cholesky decomposition for symmetric matrices over finite fields},
Finite Fields Appl. \textbf{117} (2027), 102888.

\bibitem{Weil}
A.~Weil,
\emph{On some exponential sums},
Proc. Natl. Acad. Sci. USA 34 (1948), 204--207.

\bibitem{YipClique}
C.~H.~Yip,
\emph{Exact values and improved bounds on the clique number of cyclotomic graphs},
Des. Codes Cryptogr. 93 (2025), 5131--5142.

\bibitem{YipPaleyII}
C.~H.~Yip,
\emph{Van Lint--MacWilliams' conjecture and maximum cliques in Cayley graphs over finite fields, II},
J. Combin. Theory Ser. A 221 (2026), Paper No.~106162.

\end{thebibliography}
\end{document}